\documentclass[12pt]{article}
\usepackage[active]{srcltx}
\usepackage{amssymb,latexsym,amsmath,enumerate,verbatim,amsfonts,amsthm}
\usepackage{cite}
\usepackage[ansinew]{inputenc}
\usepackage{graphicx}
\usepackage{color}
\usepackage{url}
\usepackage[colorlinks]{hyperref}
\numberwithin{equation}{section}
\newtheorem{theorem}{Theorem}[section]
\newtheorem{corollary}[theorem]{Corollary}

\newtheorem{lemma}{Lemma}[section]

\newtheorem{remark}{Remark}[section]
\begin{document}
\title{Analytical expressions of copositivity for 4th order symmetric tensors and applications}%
	
\author{Yisheng Song\thanks{School of Mathematics and Information Science  and Henan Engineering Laboratory for Big Data Statistical Analysis and Optimal Control,
			Henan Normal University, XinXiang HeNan,  P.R. China, 453007.
			Email: songyisheng@htu.cn. This author's work was supported by
			the National Natural Science Foundation of P.R. China (Grant No.
			11571095, 11601134). His work was partially done when he was
			visiting The Hong Kong Polytechnic University.}, \quad Liqun Qi\thanks{Department of Mathematics, School of Science, Hangzhou Dianzi University, Hangzhou
310018, People's Republic of China; Department of Applied Mathematics, The Hong Kong Polytechnic University, Hung Hom,
			Kowloon, Hong Kong. Email: maqilq@polyu.edu.hk. This author's work was supported by the Hong
			Kong Research Grant Council (Grant No.  PolyU 15300715, 15301716 and 15300717)}}
	\date{\today}
	
	\maketitle
	%
	%\subjclass{47H06, 47J05, 47J25, 47H10, 47H17.}
	\vskip 4mm
\begin{quote}{\bf Abstract.}\ In particle physics, scalar potentials have to be bounded
	from below in order for the physics to make sense. The precise expressions of checking lower bound of scalar potentials are essential, which is an analytical expression of checking copositivity and positive definiteness of tensors given by such scalar potentials. Because the tensors given by general scalar potential are 4th order and symmetric, our work mainly focuses on finding precise expressions to test copositivity and positive definiteness of 4th order tensors in this paper. First of all, an analytically sufficient and necessary condition of positive definiteness is provided  for 4th order 2 dimensional symmetric tensors. For 4th order 3 dimensional symmetric tensors, we give two analytically sufficient conditions of (strictly) cpositivity by using proof technique of reducing orders or dimensions of such a tensor. Furthermore, an analytically sufficient and necessary condition of copositivity is showed for 4th order 2 dimensional symmetric tensors. We also give several distinctly analytically sufficient conditions of (strict) copositivity for 4th order 2 dimensional symmetric tensors. Finally, we apply these results to check  lower bound of scalar potentials, and to present analytical vacuum stability conditions for potentials of two real scalar fields and the Higgs boson.
	\vskip 2mm
		
	\vskip 2mm
		
\end{quote}
	\vskip 8mm

	\section{\bf Introduction}\label{}
	
	Recently, Kannike \cite{K2016,K2018} studied the vacuum stability of general scalar potentials
	of a few fields.  The most general scalar potential of $n$ real singlet scalar fields $\phi_i$
	($i=1,2,\cdots,n$) can be expressed as
	\begin{equation}\label{eq:1} V(\phi) =\sum_{i,j,k,l=1}^n \lambda_{ijkl}\phi_i\phi_j\phi_k\phi_l=\Lambda \phi^4, \end{equation}
	where $\Lambda=(\lambda_{ijkl})$ is the symmetric tensor of scalar couplings and $\phi=(\phi_1,\phi_2,\cdots,\phi_n)^\top$ is the vector of
	fields.  So, the vacuum stability of such a system is equivalent to the positivity of the polynomial \eqref{eq:1} \cite{K2016}, i.e., the positive definiteness of the tensor $\Lambda=(\lambda_{ijkl})$.  However, it is NP-hard to determine the non-negativity of a given polynomial if the degree of such a polynomial is larger than or equal to 4 \cite{HL2013,MK1987}.
A significant special case \cite{K2012}, the quartic potentials of quadratic scalar fields $\phi_i^2$
	($i=1,2,\cdots,n$) is presented by
	\begin{equation}\label{eq:2} V(\phi) =\sum_{i,j=1}^n \lambda_{ij}\phi_i^2\phi_j^2=(\phi_1^2,\phi_2^2,\cdots,\phi_n^2)^\top A(\phi_1^2,\phi_2^2,\cdots,\phi_n^2),\end{equation}
	where $A=(\lambda_{ij})$ is a symmetric matrix. Then the positivity of the polynomial \eqref{eq:2} become the strict copositivity of matrix $A$. In 2012, Kannike \cite{K2012} first obtained the vacuum stability conditions of such a special case  by means of testing copositivity of matrix. The vacuum stability conditions of the general potential of two real scalars (without or with the Higgs boson included in the potential) were obtained  in \cite{K2016,K2018} with the help of the copositivity of matrix and the positivity of the polynomial.
	
The concept of copositive matrix was introduced by Motzkin \cite{TSM} in 1952. A real symmetric matrix $A$ is said to be (i) {\em copositive} if $x^TAx\geq0$ for all vector $x\geq 0$ in the non-negative orthant $\mathbb{R}^n_+$ ($x\geq 0$ implies that $x_i\geq0$ for each $i=1,2,\cdots,n$); (ii) {\em strictly copositive} if $x^TAx>0$ for all nonzero vector $x\geq 0$.  Hadeler \cite{H1983} and Nadler \cite{N1992} showed the copositive conditions of an $2\times 2$ matrix $A$ (also see Andersson-Chang-Elfving \cite{ACE}).  A real symmetric
$2\times 2$ matrix $A=(a_{ij})$ is (strictly) copositive  if and only if
$$a_{11}\geq 0 (>0), a_{22}\geq 0 (>0), a_{12}+\sqrt{a_{11}a_{22}}\geq 0 (>0).$$ The copositive conditions of an $3\times 3$ matrix $A$ were obtained by
Hadeler \cite{H1983} and  Chang-Sederberg \cite{CS1994}. A real symmetric
$3\times 3$ matrix $A=(a_{ij})$ is  copositive  if and only if
$$\begin{aligned}
a_{11}\geq 0, a_{22}\geq 0, a_{33}\geq 0,\\
 \alpha=a_{12}+\sqrt{a_{11}a_{22}}\geq 0, \beta=a_{13}+\sqrt{a_{11}a_{33}}\geq 0, \gamma=a_{23}+\sqrt{a_{33}a_{22}}\geq 0,\\
 a_{12}\sqrt{a_{33}}+a_{13}\sqrt{a_{22}}+a_{23}\sqrt{a_{11}}+\sqrt{a_{11}a_{22}a_{33}}+\sqrt{2\alpha\beta\gamma}\geq 0.
\end{aligned}$$
Ping-Yu \cite{PY1993} gave the criteria of  $4\times 4$ copositive matrices, which expression is not simpler than the above. At the same time, they proved an equivalent condition of $n\times n$ copositive matrices. Cottle-Habetler-Lemke \cite{CHL} presented analytical conditions of copositive matrix by means of the determinant and the adjugate of such a matrix. V${\ddot{a}}$li${\dot{a}}$ho \cite{HV} discussed the criteria of (strictly) copositive matrices with the help of some behaviors of its
principal submatrices. Kaplan \cite{WK} proved a way to test copositivity of a matrix by using eigenvalues and eigenvectors of its principal submatrices. Haynsworth-Hoffman \cite{HH} showed the Perron  properties of a class of copositive matrices. Johnson-Reams \cite{JR} dicussed spectral theory of copositive matrices. For more copositive properties and their applications such as copositive programs, see \cite{B09,B12,BD11,DM} and the relevant literature on this topic.
	
Recently, Kannike \cite{K2016,K2018} gave another physical example defined by scalar dark matter stable
under a $\mathbb{Z}_3$ discrete group. The most general scalar quartic
potential of the {\bf SM} Higgs $\mathbf{H}_1$, an inert doublet $\mathbf{H}_2$ and a
complex singlet $\mathbf{S}$ is
\begin{align}
V(h_1,h_2,s)=V(\phi)
=   \mathcal{V}\phi^4=&\sum_{i,j,k,l=1}^3 v_{ijkl}\phi_i\phi_j\phi_k\phi_l,\label{eq:12}
\end{align}
where $\phi=(\phi_1,\phi_2,\phi_3)^\top=(h_1,h_2,s)^\top$, $
h_1=|H_1|,\ h_2=|H_2|, H_2^{\dagger} H_1=h_1h_2\rho e^{i\phi}, S=se^{i\phi_S},$
 $\mathcal{V}=(v_{ijkl})$ is an 4th order 3 dimensional real symmetric tensor with its entries:
$v_{1111}=\lambda_1,\ v_{2222}=\lambda_2,\ v_{3333}=\lambda_S,
v_{1122}=\frac16(\lambda_3+\lambda_4\rho^2),\ v_{1133}=\frac16\lambda_{S1},\ v_{2233}=\frac16\lambda_{S2},
 v_{1233}=-\frac1{12}|\lambda_{S12}|\rho,\ v_{ijkl}=0 \mbox{ for the others}.$  Clearly, $h_1\geq0, h_2\geq0, s\geq0.$   So,  the vacuum stability of for $\mathbb{Z}_3$ scalar dark matter $V(h_1,h_2,s)$ is really equivalent to the (strict) copositivity of the tensor $\mathcal{V}=(v_{ijkl})$ (\cite{K2016,K2018}).

An $m$th order $n$ dimensional real symmetric tensor $\mathcal{A}$ is said to be \begin{itemize}
		\item[(i)] {\em copositive } if $\mathcal{A}x^m=x^T(\mathcal{A}x^{m-1})=\sum\limits_{i_1,i_2,\cdots,i_m=1}^na_{i_1i_2\cdots
		i_m}x_{i_1}x_{i_2}\cdots
	x_{i_m}\geq0$ for all $x\in \mathbb{R}^n_+$;
\item[(ii)] {\em strictly copositive} if  $\mathcal{A}x^m>0$ for all $x\in \mathbb{R}^n_+\setminus\{0\}$;
\item[(iii)] {\em semipositive definite} if  $\mathcal{A}x^m\geq0$ for all $x\in \mathbb{R}^n$ and even number $m$;
		\item[(iv)] {\em positive definite} if  $\mathcal{A}x^m>0$ for all $x\in \mathbb{R}^n\setminus\{0\}$ and even number $m$.\end{itemize}
 These concepts  were first introduced by Qi \cite{LQ1,LQ5} for higher order symmetric tensors. Qi \cite{LQ1} showed a even order  symmetric tensor is positive definitive if and only if its all H-(Z-)eigenvalues are positive. Qi \cite{LQ5} proved a symmetric tensor is strictly copositive if its each sum of the main diagonal element and negative elements in the same row is positive. Song-Qi \cite{SQ2015} extended Kaplan's test way of copositive matrix \cite{WK}  to copositive tensors, and presented some structured properties of such a class of tensors. Recently, checking copositivity of tensors has attracted the attention of  mathematical workers.  For example, Chen-Huang-Qi \cite{CHQ2017} studied some basic theory of copositivity detection of symmetric tensors and gave corresponding numerical algorithms of testing copositivity based on the standard simplex and simplicial partitions; Chen-Huang-Qi \cite{CHQ2018} revised  algorithm with a proper convex subcone of the copositive tensor cone; Nie-Yang-Zhang \cite{NYZ2018} proposed a complete semidefinite relaxation algorithm for detecting the copositivity of a symmetric tensor and showed such a detection can be done by solving a finite number of semidefinite relaxations for all tensors; Li-Zhang-Huang-Qi \cite{LZHQ2019} presented an SDP relaxation algorithm  to test the copositivity of higher order tensors. For more structured properties and numerical algorithms of copositive tensors, see \cite{CW2018, QCC2018, QL2017}.

 On the other hand, some structured tensors are closely tied to strictly copositive tensors. Song-Qi \cite{S-Q2016}  analyzed qualitatively the relationship between the constrained minimization problem on the unit sphere and (strict) copositvity of corresponding tensors. Song-Qi \cite{SQ2016} proved that a symmetric tensor is (strictly) copositive if and only if it is (strictly) semipositive. A tensor is called (strictly) semipositive if for each nonzero vector $x=(x_1,x_2,\cdots,x_n)^\top\geq0$, there exists an index $k\in \{1,2,\cdots,n\}$ such that $$x_k>0\mbox{ and }\left(\mathcal{A} x^{m-1}\right)_k=\sum_{i_2,\cdots,i_m=1}^na_{ki_2\cdots i_m}x_{i_2}\cdots x_{i_m}\geq0(>0).$$ This notion is firstly used by Song-Qi \cite{SQ2017}.  In particular, this class of tensors assure the solvability of the corresponding tensor complementarity problems (for short, TCP). So, we may probe into checking copositivity of tensors and its applications by means of studying this class of semipositive tensors. For its more properties and applications in TCP, see \cite{BHW2016,BP2018,CQS2018,CQW2016,CQW2019,DLQ2018,G2017,HQ2017,LQX2017,SY2016,S-Q2017,SM2018,WCW2018,WHB2016,WHH2018} and references cited therein.\\
	
	Until  now, there is no an analytical expression of checking copositivity and positive definitieness of tensors like ones of $2\times2$ and $3\times3$ matrices. However, the practical matters such as the vacuum stability of general scalar potentials of a few fields require precise expressions. \\

Motivated by these works above, we study the analytical expressions of certifying a symmetric tensor to be copositive and positive definitive in this paper. At the same time, we confine our work  to 4th order tensor in this paper since the tensor given by general scalar potential is 4th order. More precisely,  we provide respectively analytical expressions of testing copositivity and positive definiteness for 4th order 3 (or 2) dimensional symmetric tensors. In particular, we will employ argumentation technique of reducing orders or dimensions of such a tensor to establish the desired conclusions, which may be a very important method of analysing higher order tensors in future. Furthermore, these results can be applied to check the vacuum stability of general scalar potentials of two real singlet scalar fields and vacuum stability for $\mathbb{Z}_3$ scalar dark matter.
	
	\section{\bf Preliminaries and Basic facts}
	
An 4th order $3$ dimensional real tensor $\mathcal{A}$ consists of $81$ entries in the real field $\mathbb{R}$, i.e.,
	$$\mathcal{A} = (a_{ijkl}),\ \ \ \ \  a_{ijkl} \in \mathbb{R},\ \  i,j,k,l=1,2,3.$$
 Let the transposition of a vector $x$ be  denoted by $x^\top$. For $x = (x_1, x_2,x_3)^\top\in \mathbb{R}^3$,  $\mathcal{A}x^3$ is a vector in $\mathbb{R}^3$,
	\begin{equation}\label{eq:21}\mathcal{A}x^3=\left(\sum_{j,k,l=1}^3a_{1jkl}x_jx_k x_l,\sum_{j,k,l=1}^3a_{2jkl}x_jx_k x_l,\sum_{j,k,l=1}^3a_{3jkl}x_jx_k x_l\right)^\top.\end{equation}
Then $x^\top(\mathcal{A}x^3)$ is a homogeneous polynomial, denoted
	as $\mathcal{A}x^4$, i.e.,
	\begin{equation}\label{eq:22}\mathcal{A}x^4=x^\top(\mathcal{A}x^3)=\sum_{i,j,k,l=1}^3a_{ijkl}x_ix_jx_k
	x_l\index{$\mathcal{A}x^m$}.\end{equation}
Similarly, an 4th order $2$ dimensional real tensor $\mathcal{A}$ consists of $16$ entries in the real field $\mathbb{R}$ and for $x = (x_1, x_2)^\top\in \mathbb{R}^2$,
\begin{equation}\label{eq:2-2}\mathcal{A}x^4=x^\top(\mathcal{A}x^3)=\sum_{i,j,k,l=1}^2a_{ijkl}x_ix_jx_k
	x_l.\end{equation}
 A  tensor $\mathcal{A}$ is said to be {\em symmetric} if its entries $a_{ijkl}$ are invariant for any permutation of its indices. Obviously,   each 4th order 2 dimensional symmetric tensor $\mathcal{A}$ determines a homogeneous polynomial $\mathcal{A}x^4$ of degree 4 with 2 variables and vice versa.

Let $\|\cdot\|$ denote any norm on $\mathbb{R}^n$. Then the equivalent definition of (strict) copositivity and semipositive (positive) definiteness of a symmetric tensor in the sense of any norm on $\mathbb{R}^n$ \cite{SQ2015,LQ1,QL2017,QCC2018}.

	\begin{lemma}(\cite{SQ2015,LQ1}) \label{le:21} Let $\mathcal{A}$ be a symmetric tensor of order 4. Then
		\begin{itemize}
			\item[(i)] $\mathcal{A}$ is copositive if and only if  $\mathcal{A}x^4\geq0$ for all $x\in \mathbb{R}^n_+$ with $\|x\|=1$;
			\item[(ii)] $\mathcal{A}$ is strictly copositive if and only if $\mathcal{A}x^4>0$ for all $x\in \mathbb{R}^n_+$ with $\|x\|=1$;
			\item[(iii)] $\mathcal{A}$ is semipositive definite if and only if $\mathcal{A}x^4\geq0$ for all $x\in \mathbb{R}^n$ with $\|x\|=1$;
            \item[(iv)] $\mathcal{A}$ is positive definite if and only if $\mathcal{A}x^4>0$ for all $x\in \mathbb{R}^n$ with $\|x\|=1$.
		\end{itemize}
	\end{lemma}

A quadratic Bernstein-Bezier polynomial $p(t)$ on the interval $[0,1]$ is given by
\begin{equation}\label{eq:23}p(t)=at^2+2b(1-t)t+c(1-t)^2,\ t\in[0,1],\end{equation}
Nadler \cite{N1992} and Andersson-Chang-Elfving \cite{ACE} showed the following famous conclusion, independently.
\begin{lemma}(\cite[Lemma 1]{N1992},\cite[Lemma 2.1]{ACE}) \label{le:22} Let a quadratic Bernstein-Bezier polynomial $p(t)$ be defined by \eqref{eq:23}. Then $p(t)\geq0$ $(>0)$ for all $t\in[0,1]$ if and only if the inequalities
\begin{equation}\label{eq:24}a\geq0 (>0),\ c\geq0 (>0),\ b+\sqrt{ac}\geq0 (>0)\end{equation}
hold simultaneously.
			\end{lemma}

For a quartic and univariate polynomial $f(t)$ with real coefficients,
\begin{equation}\label{eq:25}f(t)=a_0t^4+4a_1t^3+6a_2t^2+4a_3t+a_4,\end{equation}
Gadenz-Li \cite{GL1964}, Ku \cite{K1965}, Jury-Mansor \cite{JM1981} obtained independently its positive conditions.

\begin{lemma}(\cite{GL1964,K1965,JM1981}) \label{le:23} Let $f(t)$ be a quartic and univariate polynomial defined by \eqref{eq:25} with $a_0>0$ and $a_4>0$. Define
$$\begin{aligned}
F&=9a_0^2a_2^2-24a_0a_1^2a_2+12a_1^4-a_0^3a_4+4a_0^2a_1a_3,\\
G&=a_0^2a_3-3a_0a_1a_2+2a_1^3,\\
H&=a_0a_2-a_1^2,\\
I&=a_0a_4-4a_1a_3+3a_2^2,\\
J&=a_0a_2a_4+2a_1a_2a_3-a_2^3-a_0a_3^2-a_1^2a_4,\\
\Delta&=I^3-27J^2.
\end{aligned}$$
Then $f(t)>0$ for all $0<|t|<\infty$ if and only if
\begin{itemize}
	\item[(1)]\ \ \ $ \Delta >0,\ \ \ H\geq0;$
	\item[(2)]\ \ \ $ \Delta >0,\ \ \ H<0, \ \ F<0;$
	\item[(3)]\ \ \ $ \Delta=0,\ \ \ H>0, \ \ \ F=0,\ \ \ G=0.$
\end{itemize}
\end{lemma}

For a quartic and univariate polynomial $g(t)$ with real coefficients,
\begin{equation}\label{eq:26}g(t)=at^4+bt^3+ct^2+dt+e,\end{equation}
Ulrich-Watson \cite{UW1994} proved its nonnegative conditions for $t>0$.

\begin{lemma}(\cite[Theorem 2]{UW1994}) \label{le:24} Let $g(t)$ be a quartic and univariate polynomial defined by \eqref{eq:26} with $a>0$ and $e>0$. Define
$$\begin{aligned}
\alpha &=ba^{-\frac34}e^{-\frac14},\ \beta=ca^{-\frac12}e^{-\frac12},\ \gamma=da^{-\frac14}e^{-\frac34},\\
\Delta&=4(\beta^2-3\alpha\gamma+12)^3-(72\beta+9\alpha\beta\gamma-2\beta^3-27\alpha^2-27\gamma^2)^2,\\
\mu &=(\alpha-\gamma)^2-16(\alpha+\beta+\gamma+2),\\
\eta&=(\alpha-\gamma)^2-\frac{4(\beta+2)}{\sqrt{\beta-2}}(\alpha+\gamma+4\sqrt{\beta-2}).
\end{aligned}$$
Then (i) $g(t)\geq0$ for all $t>0$ if and only if
\begin{itemize}
	\item[(1)]\ \ \ $\beta<-2\ \mbox{ and }\ \Delta \leq0\ \mbox{ and }\ \alpha+\gamma>0;$
	\item[(2)]\ \ \ $-2\leq \beta\leq 6\ \mbox{ and } \begin{cases} \Delta\leq 0\ &\mbox{ and }\ \alpha+\gamma>0\\ & or \\ \Delta\geq 0\ &\mbox{ and }\ \mu\leq 0\end{cases}$\\
	\item[(3)]\ \ \ $\beta>6\ \mbox{ and } \begin{cases} \Delta\leq 0\ &\mbox{ and }\ \alpha+\gamma>0\\ & or \\ \alpha>0\ &\mbox{ and }\ \gamma>0\\ & or  \\ \Delta\geq 0\ &\mbox{ and }\eta\leq 0.\end{cases}$
\end{itemize}
(ii) $g(t)\geq0$ for all $t>0$ if
\begin{itemize}
	\item[(1)]\ \ \ \ $\alpha>-\frac{\beta+2}{2} \mbox{ and }\gamma>-\frac{\beta+2}{2} \mbox{ for }\beta\leq6;$
\item[(2)]\ \ \ \ $\alpha>-2\sqrt{\beta-2} \mbox{ and }\gamma>-2\sqrt{\beta-2} \mbox{ for }\beta>6.$
\end{itemize}
\end{lemma}

A quadratic and multivariate polynomial $F(1-t,tv,tw)$ is difined by
\begin{equation}\label{eq:27}
\begin{aligned}F(1-t,tv,tw)=&A(1-t)^2+2(bw+cv)t(1-t)\\ &
 \ +(Bv^2+2awv+Cw^2)t^2,\ t,v\in[0,1],\end{aligned}\end{equation}
where $w=1-v.$ Chang-Sederberg \cite{CS1994} provided the following famous conclusion. Also see Nadler \cite{N1992}.
\begin{lemma}(\cite[Theorem 1]{CS1994}) \label{le:25} Let a quadratic form $F(1-t,tv,tw)$ be defined by \eqref{eq:27}.
Then $F(1-t,tv,tw)\geq0$ $(>0)$ for all $t\in[0,1]$ and all $v\in[0,1]$ and $w=1-v$ if and only if the inequalities
\begin{equation}\label{eq:28}
\begin{aligned}
&A\geq0 (>0),\ B\geq0 (>0), \ C\geq0 (>0),\\
&a+\sqrt{BC}\geq0 (>0),\ b+\sqrt{AC}\geq0 (>0),\ c+\sqrt{AB}\geq0 (>0),\\
&\sqrt{ABC}+a\sqrt{A}+b\sqrt{B}+c\sqrt{C}\\&\ \ \ \ +\sqrt{2(a+\sqrt{BC})(b+\sqrt{AC})(c+\sqrt{AB})}\geq0 (>0).
\end{aligned}
\end{equation}
hold simultaneously.
\end{lemma}

\section{\bf Copositivity of 4th order symmetric tensors}

Let $\mathcal{A}$ be an 4th order $2$ dimensional symmetric  tensor. Then for a vector $x=(x_1,x_2)^\top$,
\begin{equation}\label{eq:31}\begin{aligned}\mathcal{A}x^4&=\sum_{i,j,k,l=1}^2a_{ijkl}x_ix_jx_kx_l\\
&=a_{1111}x_1^4+4a_{1211}x_1^3x_2+6a_{1221}x_1^2x_2^2+4a_{1222}x_1x_2^3+a_{2222}x_2^4.\end{aligned}\end{equation}
Take $y=(1,0)^\top$ and $z=(0,1)^\top$. Then  $\mathcal{A}y^4=a_{1111}$ and $\mathcal{A}z^4=a_{2222}$. So, it is obvious that $a_{1111}>0$ and $a_{2222}>0$ are necessary condition of positive definiteness of $\mathcal{A}$.

\begin{theorem} \label{th:31} Let $\mathcal{A}$ be a symmetric tensor of order $4$ and dimension $2$ with $a_{1111}>0$ and $a_{2222}>0$.
	Then  $\mathcal{A}$ is positive definite if and only if
	\begin{itemize}
\item[(1)]\ $ a_{1111}a_{1221}\geq a_{1211}^2,\\ (a_{1111}a_{2222}-4a_{1211}a_{1222}+3a_{1221}^2)^3>27(a_{1111}a_{1221}a_{2222}+2a_{1211}a_{1221}a_{1222}-a_{1221}^3-a_{1111}a_{1222}^2-a_{1211}^2a_{2222})^2;$
\item[(2)]\ $ a_{1111}a_{1221}< a_{1211}^2,\\ (a_{1111}a_{2222}-4a_{1211}a_{1222}+3a_{1221}^2)^3>27(a_{1111}a_{1221}a_{2222}+2a_{1211}a_{1221}a_{1222}-a_{1221}^3-a_{1111}a_{1222}^2-a_{1211}^2a_{2222})^2,\\ 9a_{1111}^2a_{1221}^2+12a_{1211}^4+4a_{1111}^2a_{1211}a_{1222}<a_{1111}^3a_{2222}+24a_{1111}a_{1211}^2a_{1221};$
\item[(3)]\ $ a_{1111}a_{1221}> a_{1211}^2,\\ (a_{1111}a_{2222}-4a_{1211}a_{1222}+3a_{1221}^2)^3=27(a_{1111}a_{1221}a_{2222}+2a_{1211}a_{1221}a_{1222}-a_{1221}^3-a_{1111}a_{1222}^2-a_{1211}^2a_{2222})^2,\\ 9a_{1111}^2a_{1221}^2+12a_{1211}^4 +4a_{1111}^2a_{1211}a_{1222}=a_{1111}^3a_{2222}+24a_{1111}a_{1211}^2a_{1221},\\ a_{1111}^2a_{1222}+2a_{1211}^3=3a_{1111}a_{1211}a_{1221}.$
	\end{itemize}
\end{theorem}
	
	\begin{proof}  It follows from Lemma \ref{le:21} that we can restrict $x$ to $$\|x\|=|x_1|+|x_2|=1.$$
	Consider the homogeneous polynomial $\mathcal{A}x^4$ with $a_{1111}>0$ and $a_{2222}>0$	in three cases.
	
	Case 1. $x_1=0$ and $x_2\ne0$. Then $|x_2|=1$, and hence, $\mathcal{A}x^4=a_{2222}>0.$
	
	Case 2. $x_1\ne0$ and $x_2=0$. Then $|x_1|=1$, and hence, $\mathcal{A}x^4=a_{1111}>0.$
	
	Case 3. $x_1\ne0$ and $x_2\ne0$. Then the homogeneous polynomial $\mathcal{A}x^4$ can be divided by $x_2^4$ to  yield $$\frac{\mathcal{A}x^4}{x_2^4}=a_{1111}\left(\frac{x_1}{x_2}\right)^4+4a_{1211}\left(\frac{x_1}{x_2}\right)^3+6a_{1221}\left(\frac{x_1}{x_2}\right)^2+4a_{1222}\left(\frac{x_1}{x_2}\right)+a_{2222}.$$
Let $t=\frac{x_1}{x_2}$ and $f(t)=\frac{\mathcal{A}x^4}{x_2^4}$, i.e., \begin{equation}\label{eq:32}f(t)=a_{1111}t^4+4a_{1211}t^3+6a_{1221}t^2+4a_{1222}t+a_{2222}.\end{equation}
	Clearly, $f(t)>0$ if and only if $\mathcal{A}x^4>0$. Assume that
	$$\begin{aligned}
	F&=9a_{1111}^2a_{1221}^2-24a_{1111}a_{1211}^2a_{1221}+12a_{1211}^4-a_{1111}^3a_{2222} +4a_{1111}^2a_{1211}a_{1222},\\
	G&=a_{1111}^2a_{1222}-3a_{1111}a_{1211}a_{1221}+2a_{1211}^3,\\
	H&=a_{1111}a_{1221}-a_{1211}^2,\\
	I&=a_{1111}a_{2222}-4a_{1211}a_{1222}+3a_{1221}^2,\\
	J&=a_{1111}a_{1221}a_{2222}+2a_{1211}a_{1221}a_{1222}-a_{1221}^3-a_{1111}a_{1222}^2-a_{1211}^2a_{2222},\\
	\Delta&=I^3-27J^2.
	\end{aligned}$$ Therefore, the conclusions directly follow from Lemma \ref{le:23} with $a_0=a_{1111}$, $a_1=a_{1211}$, $a_2=a_{1221}$, $a_3=a_{1222}$ and $a_4=a_{2222}$, as required.
		\end{proof}

\begin{remark}	
It follows from the proof of Theorem \ref{th:31} that $\mathcal{A}x^4$ may be divided by $x_1^4$, then $$f(t)=\frac{\mathcal{A}x^4}{x_1^4}=a_{1111}+4a_{1211}t+6a_{1221}t^2+4a_{1222}t^3+a_{2222}t^4,$$ where $t=\frac{x_2}{x_1}$. So the conclusions still hold if $a_{2222}$ and $a_{1111}$ are replaced each other on the assumptions of Theorem \ref{th:31}.	
\end{remark}

	\begin{theorem} \label{th:32} Let $\mathcal{A}$ be a symmetric tensor of order $4$ and dimension $3$. Assume that
		$$\begin{aligned}
		&a_{1111}\geq0, a_{2222}\geq0, a_{3333}\geq0, a_{1122}\geq0, a_{1133}\geq0, a_{2233}\geq0,\\
		&\alpha_1=6a_{1231}+3\sqrt{a_{1122}a_{1133}}\geq0, \beta_1= 2a_{1113}+\sqrt{3a_{1111}a_{1133}}\geq0,\\
		& \gamma_1=2a_{2111}+\sqrt{3a_{1111}a_{1122}}\geq0,\alpha_2=2a_{3222}+\sqrt{3a_{2222}a_{2233}}\geq0,\\
		& \beta_2= 6a_{1223}+3\sqrt{a_{1122}a_{2233}}\geq0, \gamma_2=2a_{1222}+\sqrt{3a_{1122}a_{2222}}\geq0,\\
		&\alpha_3=2a_{2333}+\sqrt{3a_{3333}a_{2233}}\geq0, \beta_3= 2a_{1333}+\sqrt{3a_{1133}a_{3333}}\geq0,\\
			& \gamma_3=6a_{1233}+3\sqrt{2a_{1133}a_{2233}}\geq0,
\end{aligned}$$
$$\begin{aligned}
	\tau_1=&3\sqrt{a_{1111}a_{1122}a_{1133}}+6a_{1231}\sqrt{a_{1111}}+2a_{1113}\sqrt{3a_{1122}}+2a_{2111}\sqrt{3a_{1133}}\\
		&\ \ \ \ +\sqrt{2\alpha_1\beta_1\gamma_1}\geq0,\\
	\tau_2=&3\sqrt{a_{1122}a_{2222}a_{2233}}+2a_{3222}\sqrt{3a_{1122}}+6a_{1223}\sqrt{a_{2222}}+2a_{1222}\sqrt{3a_{2233}}\\
		&\ \ \ \ +\sqrt{2\alpha_2\beta_2\gamma_2}\geq0,\\
	\tau_3=&3\sqrt{a_{1133}a_{3333}a_{2233}}+2a_{2333}\sqrt{3a_{1133}}+2a_{1333}\sqrt{3a_{2233}}+6a_{1233}\sqrt{a_{3333}}\\
		&\ \ \ \ +\sqrt{2\alpha_3\beta_3\gamma_3}\geq0.
\end{aligned}$$
		Then  $\mathcal{A}$ is copositive.
	\end{theorem}
	
	\begin{proof}  It follows from Lemma \ref{le:21} that we can restrict $x$ to
		$$\|x\|=x_1+x_2+x_3=1\mbox{ and }x_i\geq0\mbox{ for }i=1,2,3.$$
	Without loss of generality, let $x_1=1-t$ and  $x_2=tv$ and  $x_3=tw$ for $t,v\in[0,1]$ and $w=1-v$.  Clearly, we have
		\begin{equation}\label{eq:32}\begin{aligned}\mathcal{A}x^4=&\sum_{i,j,k,l=1}^3a_{ijkl}x_ix_jx_kx_l\\
		=&a_{1111}(1-t)^4+a_{2222}(tv)^4+a_{3333}(tw)^4\\
		&+4a_{1222}(1-t)(tv)^3 +4a_{1333}(1-t)(tw)^3 +4a_{2111}(1-t)^3(tv)\\&+4a_{2333}(tv)(tw)^3+4a_{3111}(1-t)^3(tw)
		+4a_{3222}(tv)^3(tw)\\
		&+6a_{1122}(1-t)^2(tv)^2+6a_{1133}(1-t)^2(tw)^2+6a_{2233}(tv)^2(tw)^2\\
		&+12a_{1231}(1-t)^2(tv)(tw)+12a_{1232}(1-t)(tv)^2(tw)\\
		&+12a_{1233}(1-t)(tv)(tw)^2.
		\end{aligned}\end{equation}
Then through simple calculation to yield
	$$\begin{aligned}\mathcal{A}x^4=&[a_{1111}(1-t)^2+2(2a_{3111}w+2a_{2111}v)t(1-t)\\
	&+(3a_{1122}v^2+12a_{1231}vw+3a_{1133}w^2)t^2](1-t)^2\\
	&+[3a_{1122}(1-t)^2+2(6a_{1232}w+2a_{1222}v)t(1-t)\\
	&+(a_{2222}v^2+4a_{3222}vw+3a_{2233}w^2)t^2](tv)^2\\
	&+[3a_{1133}(1-t)^2+2(2a_{1333}w+6a_{1233}v)t(1-t)\\
	&+(3a_{2233}v^2+4a_{2333}vw+a_{3333}w^2)t^2](tw)^2.
	\end{aligned}$$	
Let $$\begin{aligned}
F_1(1-t,tv,tw)=&a_{1111}(1-t)^2+2(2a_{3111}w+2a_{2111}v)t(1-t)\\
&+(3a_{1122}v^2+12a_{1231}vw+3a_{1133}w^2)t^2; \\
F_2(1-t,tv,tw)=&3a_{1122}(1-t)^2+2(6a_{1232}w+2a_{1222}v)t(1-t)\\
&+(a_{2222}v^2+4a_{3222}vw+3a_{2233}w^2)t^2;\\
F_3(1-t,tv,tw)=&3a_{1133}(1-t)^2+2(2a_{1333}w+6a_{1233}v)t(1-t)\\
&+(3a_{2233}v^2+4a_{2333}vw+a_{3333}w^2)t^2.
\end{aligned}$$

For the function $F_1(1-t,tv,tw)$ with the assumptions that $$a_{1111}\geq0,a_{1122}\geq0,a_{1133}\geq0, \alpha_1\geq0,\beta_1\geq0,\gamma_1\geq0,\tau_1\geq0,$$
it follows from Lemma \ref{le:25} that $F_1(1-t,tv,tw)\geq0$  for all $t,v\in[0,1]$ and $w=1-v$. Similarly, we also have $F_2(1-t,tv,tw)\geq0$ and $F_3(1-t,tv,tw)\geq0$  for all $t,v\in[0,1]$ and $w=1-v$. So,
$$\mathcal{A}x^4=F_1(1-t,tv,tw)(1-t)^2+F_2(1-t,tv,tw)(tv)^2+F_3(1-t,tv,tw)(tw)^2\geq0.$$
  Therefore, $\mathcal{A}x^4\geq0$ for all $x\geq0$ and $\|x\|=1$. Namely, the tensor $\mathcal{A}$ is copositive, as required.
	\end{proof}

Obviously, if ``$\geq$" is replaced by ``$>$" in all conditions of Theorem \ref{th:32}, then the strict copositivity of   $\mathcal{A}$ can be showed easily.
	
	\begin{theorem} \label{th:33} Let $\mathcal{A}$ be a symmetric tensor of order $4$ and dimension $3$. Assume that
		$$\begin{aligned}
		&a_{1111}>0, a_{2222}>0, a_{3333}>0, a_{1122}>0, a_{1133}>0, a_{2233}>0,\\
		&\alpha_1=6a_{1231}+3\sqrt{a_{1122}a_{1133}}>0, \beta_1= 2a_{1113}+\sqrt{3a_{1111}a_{1133}}>0,\\
		& \gamma_1=2a_{2111}+\sqrt{3a_{1111}a_{1122}}>0,\alpha_2=2a_{3222}+\sqrt{3a_{2222}a_{2233}}>0,\\
		& \beta_2= 6a_{1223}+3\sqrt{a_{1122}a_{2233}}>0, \gamma_2=2a_{1222}+\sqrt{3a_{1122}a_{2222}}>0,\\
		&\alpha_3=2a_{2333}+\sqrt{3a_{3333}a_{2233}}>0, \beta_3= 2a_{1333}+\sqrt{3a_{1133}a_{3333}}>0,\\
		& \gamma_3=6a_{1233}+3\sqrt{2a_{1133}a_{2233}}>0,
		\end{aligned}$$
		$$\begin{aligned}
		\tau_1=&3\sqrt{a_{1111}a_{1122}a_{1133}}+6a_{1231}\sqrt{a_{1111}}+2a_{1113}\sqrt{3a_{1122}}+2a_{2111}\sqrt{3a_{1133}}\\
		&\ \ \ \ +\sqrt{2\alpha_1\beta_1\gamma_1}>0,\\
		\tau_2=&3\sqrt{a_{1122}a_{2222}a_{2233}}+2a_{3222}\sqrt{3a_{1122}}+6a_{1223}\sqrt{a_{2222}}+2a_{1222}\sqrt{3a_{2233}}\\
		&\ \ \ \ +\sqrt{2\alpha_2\beta_2\gamma_2}>0,\\
		\tau_3=&3\sqrt{a_{1133}a_{3333}a_{2233}}+2a_{2333}\sqrt{3a_{1133}}+2a_{1333}\sqrt{3a_{2233}}+6a_{1233}\sqrt{a_{3333}}\\
		&\ \ \ \ +\sqrt{2\alpha_3\beta_3\gamma_3}>0.
		\end{aligned}$$
		Then  $\mathcal{A}$ is strictly copositive.
		\end{theorem}
	
\begin{remark}
 From the proof of Theorems \ref{th:32} and \ref{th:33}, it is easily seen to prove the desired results by reducing orders of tensor. That is, an 4th order 3 dimensional tensor is decomposed  three 2nd order 3 dimensional tensors, and then, analysing the copositivity of these 2nd order tensors to obtain the desired sufficient conditions. This may be a very important way to studying higher tensors in future.
\end{remark}

	\begin{theorem} \label{th:34} Let $\mathcal{A}$ be a symmetric tensor of order $4$ and dimension $2$ with $a_{1111}>0$ and $a_{2222}>0$. Assume that
		$$\begin{aligned}
		(1)\	a_{1221}\leq \sqrt{a_{1111}a_{2222}},\ \ &4a_{1211}\sqrt[4]{a_{2222}}+ \sqrt[4]{a_{1111}}(3a_{1221}+\sqrt{a_{1111}a_{2222}})>0 \mbox{ and }\\& 4a_{1222}\sqrt[4]{a_{1111}}+\sqrt[4]{a_{2222}}(3a_{1221}+\sqrt{a_{1111}a_{2222}})>0; \\
		(2)\	a_{1221}>\sqrt{a_{1111}a_{2222}},\ \  &2a_{1211}+\sqrt{6a_{1221}a_{1111}-2a_{1111}\sqrt{a_{1111}a_{2222}}}>0 \ \mbox{ and }\\& 2a_{1222}+\sqrt{6a_{1221}a_{2222}-2a_{2222}\sqrt{a_{1111}a_{2222}}}>0. \\
		\end{aligned}$$
		Then $\mathcal{A}$ is copositive.
	\end{theorem}
	\begin{proof}  It follows from Lemma \ref{le:21} that we can restrict $x=(x_1,x_2)^\top$ to $$x_1\geq0,\ x_2\geq0,\ \|x\|=x_1+x_2=1.$$
		Consider the homogeneous polynomial $\mathcal{A}x^4$ with $a_{1111}>0$ and $a_{2222}>0$	in three cases.
		
		Case 1. $x_1=0$ and $x_2\ne0$. Then $x_2=1$, and hence, $\mathcal{A}x^4=a_{2222}>0.$
		
		Case 2. $x_1\ne0$ and $x_2=0$. Then $x_1=1$, and hence, $\mathcal{A}x^4=a_{1111}>0.$
		
		Case 3. $x_1\ne0$ and $x_2\ne0$. Then the homogeneous polynomial $\mathcal{A}x^4$ can be divided by $x_2^4$ to  yield $$\frac{\mathcal{A}x^4}{x_2^4}=a_{1111}\left(\frac{x_1}{x_2}\right)^4+4a_{1211}\left(\frac{x_1}{x_2}\right)^3+6a_{1221}\left(\frac{x_1}{x_2}\right)^2+4a_{1222}\left(\frac{x_1}{x_2}\right)+a_{2222}.$$
		Let $t=\frac{x_1}{x_2}$ and $g(t)=\frac{\mathcal{A}x^4}{x_2^4}$, i.e., \begin{equation}\label{eq:32}g(t)=a_{1111}t^4+4a_{1211}t^3+6a_{1221}t^2+4a_{1222}t+a_{2222}.\end{equation}
		Clearly, $g(t)\geq0$ if and only if $\mathcal{A}x^4\geq0$. Let $$	\alpha=4a_{1211}a_{1111}^{-\frac34}a_{2222}^{-\frac14},\ \beta=6a_{1221}a_{1111}^{-\frac12}a_{2222}^{-\frac12},\ \gamma=4a_{1222}a_{1111}^{-\frac14}a_{2222}^{-\frac34}.$$
	Then for the aussumption (1), the inquality	$a_{1221}\leq \sqrt{a_{1111}a_{2222}}$ means that $$a_{1221}a_{1111}^{-\frac12}a_{2222}^{-\frac12}\leq1, \mbox{ i.e., }\beta\leq6;$$
Multiply the inquality $4a_{1211}\sqrt[4]{a_{2222}}+ \sqrt[4]{a_{1111}}(3a_{1221}+\sqrt{a_{1111}a_{2222}})>0$  by $a_{1111}^{-\frac34}a_{2222}^{-\frac12}$ to yield $$4a_{1211}a_{1111}^{-\frac34}a_{2222}^{-\frac14}+ a_{1111}^{-\frac12}a_{2222}^{-\frac12}(3a_{1221}+\sqrt{a_{1111}a_{2222}})>0. $$
Namely, $$\alpha=4a_{1211}a_{1111}^{-\frac34}a_{2222}^{-\frac14}>- (3a_{1221}a_{1111}^{-\frac12}a_{2222}^{-\frac12}+1)=-\frac{\beta+2}{2}. $$	
Similarly, multiply the inquality $4a_{1222}\sqrt[4]{a_{1111}}+ \sqrt[4]{a_{2222}}(3a_{1221}+\sqrt{a_{1111}a_{2222}})>0$ by $a_{1111}^{-\frac12}a_{2222}^{-\frac34}$ to yield $$\gamma=4a_{1222}a_{1111}^{-\frac14}a_{2222}^{-\frac34}>-(3a_{1221}a_{1111}^{-\frac12}a_{2222}^{-\frac12}+1)=-\frac{\beta+2}{2}.$$
This means that $$\alpha>-\frac{\beta+2}{2} \mbox{ and }\gamma>-\frac{\beta+2}{2} \mbox{ for }\beta\leq6.$$ Likewise, the assumptions  (2) imply that $$\alpha>-2\sqrt{\beta-2} \mbox{ and }\gamma>-2\sqrt{\beta-2} \mbox{ for }\beta>6.$$
So the conclusions directly follow from Lemma \ref{le:24} (ii), as required.
	\end{proof}	
Similarly, using Lemma \ref{le:24} (i), the following conclusion is obtained easily.

\begin{theorem} \label{th:35} Let $\mathcal{A}$ be a symmetric tensor of order $4$ and dimension $2$ with $a_{1111}>0$ and $a_{2222}>0$.
	Then $\mathcal{A}$ is copositive if and only if
	\begin{itemize}
	\item[(1)]\ $a_{1221}<-\frac{1}{3}\sqrt{a_{1111}a_{2222}},\ a_{1211}\sqrt{a_{2222}}+a_{1222}\sqrt{a_{1111}}>0,\\ (3a_{1221}^2-4a_{1211}a_{1222}+a_{1111}a_{2222})^3\leq27(a_{1111}a_{1221}a_{2222}+2a_{1211}a_{1221}a_{1222}-a_{1221}^3-a_{1111}a_{1222}^2-a_{1211}^2a_{2222})^2;$
	\item[(2)]\ $-\frac{1}{3}\sqrt{a_{1111}a_{2222}}\leq a_{1221}\leq \sqrt{a_{1111}a_{2222}}\ \mbox{ and }\\
 \begin{cases}(3a_{1221}^2-4a_{1211}a_{1222}+a_{1111}a_{2222})^3\leq 27(a_{1111}a_{1221}a_{2222}\\ +2a_{1211}a_{1221}a_{1222} -a_{1221}^3-a_{1111}a_{1222}^2-a_{1211}^2a_{2222})^2, \\ a_{1211}\sqrt{a_{2222}}+a_{1222}\sqrt{a_{1111}}>0,\\
 \ \ \ \ \mbox{\ \ or} \\
 (3a_{1221}^2-4a_{1211}a_{1222}+a_{1111}a_{2222})^3\geq 27(a_{1111}a_{1221}a_{2222}+2a_{1211}a_{1221}a_{1222}\\
 -a_{1221}^3-a_{1111}a_{1222}^2-a_{1211}^2a_{2222})^2,\\ a_{1211}^2a_{2222}-2a_{1211}a_{1222}\sqrt{a_{1111}a_{2222}}+a_{1222}^2a_{1111}\leq 4a_{1211}a_{1111}^{\frac34}a_{2222}^{\frac54}\\+6a_{1111}a_{1211}a_{2222}+4a_{1222}a_{1111}^{\frac54}a_{2222}^{\frac34}+2a_{1111}^{\frac32}a_{2222}^{\frac32}
 \end{cases}$
	\item[(3)]\ $a_{1221}>\sqrt{a_{1111}a_{2222}}\ \mbox{ and }\\
 \begin{cases} (3a_{1221}^2-4a_{1211}a_{1222}+a_{1111}a_{2222})^3\leq 27(a_{1111}a_{1221}a_{2222}\\ +2a_{1211}a_{1221}a_{1222} -a_{1221}^3-a_{1111}a_{1222}^2-a_{1211}^2a_{2222})^2, \\ a_{1211}\sqrt{a_{2222}}+a_{1222}\sqrt{a_{1111}}>0,\\
 \ \ \ \ \mbox{\ \ or} \\  a_{1211}>0\ \mbox{ and }\ a_{1222}>0\\
  \ \ \ \ \mbox{\ \ or} \\
  (3a_{1221}^2-4a_{1211}a_{1222}+a_{1111}a_{2222})^3\geq 27(a_{1111}a_{1221}a_{2222}\\+2a_{1211}a_{1221}a_{1222}
 -a_{1221}^3-a_{1111}a_{1222}^2-a_{1211}^2a_{2222})^2,\\
 (a_{1211}\sqrt{a_{2222}}-a_{1222}\sqrt{a_{1111}})^2\sqrt{6a_{1221}-2\sqrt{a_{1111}a_{2222}}}\\\leq 6a_{1221}+2\sqrt{a_{1111}a_{2222}})(a_{1211}a_{2222}\sqrt{a_{1111}}\\+a_{1222}a_{1111}\sqrt{a_{2222}}+a_{2222}a_{1111}\sqrt{6a_{1221}-2\sqrt{a_{1111}a_{2222}}}.
 \end{cases}$
	\end{itemize}
	\end{theorem}
\begin{proof} Using the similar proof technique of Theorem \ref{th:34}, we only need consider the nonnegativity of the polynomial,
		$$g(t)=a_{1111}t^4+4a_{1211}t^3+6a_{1221}t^2+4a_{1222}t+a_{2222}$$
where $t=\frac{x_1}{x_2}$. Let $$\alpha=4a_{1211}a_{1111}^{-\frac34}a_{2222}^{-\frac14},\ \beta=6a_{1221}a_{1111}^{-\frac12}a_{2222}^{-\frac12},\ \gamma=4a_{1222}a_{1111}^{-\frac14}a_{2222}^{-\frac34}.$$
Then $$\begin{aligned}
	      \Delta=&4(\beta^2-3\alpha\gamma+12)^3-(72\beta+9\alpha\beta\gamma-2\beta^3-27\alpha^2-27\gamma^2)^2,\\
	            =&\frac{4\times12^3}{a_{1111}^3a_{2222}^3}[(3a_{1221}^2-4a_{1211}a_{1222}+a_{1111}a_{2222})^3- 27(a_{1111}a_{1221}a_{2222}\\
&\ \  +2a_{1211}a_{1221}a_{1222} -a_{1221}^3-a_{1111}a_{1222}^2-a_{1211}^2a_{2222})^2], \\
	     \end{aligned}$$
$$\begin{aligned}
\mu=&(\alpha-\gamma)^2-16(\alpha+\beta+\gamma+2),\\
=&\frac{16}{(a_{1111}a_{2222})^{\frac32}} [(a_{1211}\sqrt{a_{2222}}-a_{1222}\sqrt{a_{1111}})^2-( 4a_{1211}a_{1111}^{\frac34}a_{2222}^{\frac54} \\
&+6a_{1111}a_{1211}a_{2222}+4a_{1222}a_{1111}^{\frac54}a_{2222}^{\frac34}+2a_{1111}^{\frac32}a_{2222}^{\frac32})]\\
 =&\frac{16}{(a_{1111}a_{2222})^{\frac32}} [(a_{1211}^2a_{2222}-2a_{1211}a_{1222}\sqrt{a_{1111}a_{2222}}+a_{1222}^2a_{1111})\\
 &-( 4a_{1211}a_{1111}^{\frac34}a_{2222}^{\frac54} +6a_{1111}a_{1211}a_{2222}+4a_{1222}a_{1111}^{\frac54}a_{2222}^{\frac34}+2a_{1111}^{\frac32}a_{2222}^{\frac32})],
\end{aligned}$$
$$\begin{aligned}
\eta=&(\alpha-\gamma)^2-\frac{4(\beta+2)}{\sqrt{\beta-2}}(\alpha+\gamma+4\sqrt{\beta-2})\\
    =&\frac{16}{(a_{1111}a_{2222})^{\frac74}\sqrt{\beta-2}}[(a_{1211}\sqrt{a_{2222}}-a_{1222}\sqrt{a_{1111}})^2\sqrt{6a_{1221}-2\sqrt{a_{1111}a_{2222}}}\\
    &- (6a_{1221}+2\sqrt{a_{1111}a_{2222}})(a_{1211}a_{2222}\sqrt{a_{1111}}+a_{1222}a_{1111}\sqrt{a_{2222}}\\
    &+a_{2222}a_{1111}\sqrt{6a_{1221}-2\sqrt{a_{1111}a_{2222}}})].
\end{aligned}$$
Thus, the assumption (1) means that (using the inequality $a_{1221}\leq -\frac13\sqrt{a_{1111}a_{2222}}$) $$\beta=6a_{1221}a_{1111}^{-\frac12}a_{2222}^{-\frac12}\leq 6\times(-\frac13\sqrt{a_{1111}a_{2222}})a_{1111}^{-\frac12}a_{2222}^{-\frac12}=-2, $$ $\Delta\leq0$ and
$$\alpha+\gamma=\frac{4}{(a_{1111}a_{2222})^\frac34} (a_{1211}\sqrt{a_{2222}}+a_{1222}\sqrt{a_{1111}})>0.$$
Similarly, by using a simple calculation, we also have\begin{itemize}
		\item[(2)]\ \ \ $-2\leq \beta\leq 6\ \mbox{ and } \begin{cases} \Delta\leq 0\ &\mbox{ and }\ \alpha+\gamma>0\\ & or \\ \Delta\geq 0\ &\mbox{ and }\ \mu\leq 0\end{cases}$\\
	\item[(3)]\ \ \ $\beta>6\ \mbox{ and } \begin{cases} \Delta\leq 0\ &\mbox{ and }\ \alpha+\gamma>0\\ & or \\ \alpha>0\ &\mbox{ and }\ \gamma>0\\ & or  \\ \Delta\geq 0\ &\mbox{ and }\eta\leq 0.\end{cases}$
\end{itemize}
Thus, the conclusions directly follow from Lemma \ref{le:24} (i), as required.
	\end{proof}	
Now we give several simpler sufficient conditions of (strictly) copositive tensors	

\begin{theorem} \label{th:36} Let $\mathcal{A}$ be a symmetric tensor of order $4$ and dimension $2$. Assume that
	$$\begin{aligned}
	&\ a_{1111}\geq 0\ (>0),a_{2222}\geq 0\ (>0), 	a_{1112}\geq 0\ (>0) ,\ a_{2221}\geq0(>0),\\ &3a_{1221}+\sqrt{a_{1111}a_{2222}}+4\sqrt{a_{2111}a_{1222}} \geq 0\ (>0).
	\end{aligned}$$
	Then $\mathcal{A}$ is (strictly) copositive.
\end{theorem}
\begin{proof}  It follows from Lemma \ref{le:21} that we can restrict $x=(x_1,x_2)^\top$ to $$x_1\geq0,\ x_2\geq0,\ \|x\|=x_1+x_2=1.$$
Without loss of generality, we may assume $x_1=t$ and $x_2=1-t$ for all $t\in[0,1]$. Then

\begin{equation}\label{eq:35}\begin{aligned}\mathcal{A}x^4&=\sum_{i,j,k,l=1}^2a_{ijkl}x_ix_jx_kx_l\\
&=a_{1111}t^4+4a_{1112}t^3(1-t)+6a_{1221}t^2(1-t)^2\\
&\ \ \ \ \ +4a_{2221}t(1-t)^3+a_{2222}(1-t)^4.\end{aligned}
\end{equation}
 For $t\in(0,1)$,
rewritten \eqref{eq:35} as follow,
$$\begin{aligned}
\mathcal{A}x^4&=\left(\sqrt{a_{1111}}t^2-\sqrt{a_{2222}}(1-t)^2\right)^2\\
&\ \ \ +2t(1-t)\left(2a_{1112}t^2+(3a_{1221}+\sqrt{a_{1111}a_{2222}})t(1-t)+2a_{2221}(1-t)^2\right).
\end{aligned}$$
Let $$p(t)=2a_{1112}t^2+(3a_{1221}+\sqrt{a_{1111}a_{2222}})t(1-t)+2a_{2221}(1-t)^2.$$
Since the inquality	$3a_{1221}+\sqrt{a_{1111}a_{2222}}+4\sqrt{a_{2111}a_{1222}} \geq 0\ (>0)$ means that $$\frac{3a_{1221}+\sqrt{a_{1111}a_{2222}}}2+2\sqrt{a_{1112}a_{2221}} \geq 0\ (>0)$$
It follows from Lemma \ref{le:22} that $p(t)\geq 0\ (>0),$ and hence,
$$\mathcal{A}x^4=\left(\sqrt{a_{1111}}t^2-\sqrt{a_{2222}}(1-t)^2\right)^2+2t(1-t)p(t)\geq 0\ (>0).$$
Clearly, if $t=0$ or $t=1$, $\mathcal{A}x^4=a_{2222}$ or $a_{1111}$.
Thus, $$\mathcal{A}x^4\geq 0\ (>0) \mbox{ for all }x\geq0\ \mbox{ with }\  \|x\|=1.$$
So, the conclusions directly follow from Lemma \ref{le:21}, as required.
\end{proof}	
	
\begin{theorem} \label{th:37} Let $\mathcal{A}$ be a symmetric tensor of order $4$ and dimension $2$. Assume that
	$$\begin{aligned}
	&\ a_{1111}\geq 0\ (>0),a_{2222}\geq 0\ (>0), \\&\	a_{1112}+\sqrt[4]{a_{1111}^3a_{2222}}\geq 0\ (>0) ,\ a_{2221}+\sqrt[4]{a_{1111}a_{2222}^3}\geq0\ (>0),\\ &3(a_{1221}-\sqrt{a_{1111}a_{2222}})\\
	&\ \ +4\sqrt{\left(a_{1112}+\sqrt[4]{a_{1111}^3a_{2222}}\right)\left(a_{2221}+\sqrt[4]{a_{1111}a_{2222}^3}\right)} \geq 0\ (>0).
	\end{aligned}$$
	Then $\mathcal{A}$ is (strictly) copositive.
\end{theorem}
\begin{proof} Using the same argumentation technique,  	For $t\in(0,1)$,
rewritten \eqref{eq:35} as follow,
	$$\begin{aligned}
	\mathcal{A}x^4&=\left(\sqrt[4]{a_{1111}}t-\sqrt[4]{a_{2222}}(1-t)\right)^4+4\left(a_{1112}+\sqrt[4]{a_{1111}^3a_{2222}}\right)t^3(1-t)\\
	&\ \ \ +6\left(a_{1221}-\sqrt{a_{1111}a_{2222}}\right)t^2(1-t)^2+4\left(a_{2221}+\sqrt[4]{a_{1111}a_{2222}^3}\right)t(1-t)^3\\
	&= \left(\sqrt[4]{a_{1111}}t-\sqrt[4]{a_{2222}}(1-t)\right)^4+t(1-t)p(t),
	\end{aligned}$$
where $$\begin{aligned}p(t)=&4\left(a_{1112}+\sqrt[4]{a_{1111}^3a_{2222}}\right)t^2+6\left(a_{1221}-\sqrt{a_{1111}a_{2222}}\right)t(1-t)\\ &\ +4\left(a_{2221}+\sqrt[4]{a_{1111}a_{2222}^3}\right)(1-t)^2.\end{aligned}$$

The aussumptions assure $p(t)\geq 0\ (>0)$ for all $t\in(0,1)$ by Lemma \ref{le:22}.
It is obvious that $\mathcal{A}x^4=a_{2222}$ or $a_{1111}$ for  $t=0$ or $t=1$.
Thus, $$\mathcal{A}x^4\geq 0\ (>0) \mbox{ for all }x\geq0\ \mbox{ with }\  \|x\|=1.$$
Therefore, $\mathcal{A}$ is (strictly) copositive.
\end{proof}	

From the proving process of Theorems \ref{th:36} and \ref{th:37}, the following conclusion is proved easily.

\begin{corollary} \label{co:38} Let $\mathcal{A}$ be a symmetric and strictly copositive tensor of order $4$ and dimension $2$. Then
	\begin{equation}\label{eq:36}
2a_{1112}\sqrt{a_{2222}}+(3a_{1221}+\sqrt{a_{1111}a_{2222}})\sqrt[4]{a_{1111}a_{2222}}+2a_{1222}\sqrt{a_{1111}}>0.
\end{equation}
\end{corollary}
\begin{proof} It follows from the strict copositivity of  $\mathcal{A}$ that $a_{1111}>0$ and $a_{2222}>0.$
For $t\in(0,1)$ and $x=(t,1-t)^\top$, we  have
$$\begin{aligned}
\mathcal{A}x^4&=\left(\sqrt{a_{1111}}t^2-\sqrt{a_{2222}}(1-t)^2\right)^2\\
&\ \ \ +2t(1-t)\left(2a_{1112}t^2+(3a_{1221}+\sqrt{a_{1111}a_{2222}})t(1-t)+2a_{2221}(1-t)^2\right).
\end{aligned}$$
Take $t_0=\frac{\sqrt[4]{a_{2222}}}{\sqrt[4]{a_{1111}}+\sqrt[4]{a_{2222}}}$. Then $x^0=(t_0,1-t_0)^\top$, i.e., $$x^0=\left(\frac{\sqrt[4]{a_{2222}}}{\sqrt[4]{a_{1111}}+\sqrt[4]{a_{2222}}}, \frac{\sqrt[4]{a_{1111}}}{\sqrt[4]{a_{1111}}+\sqrt[4]{a_{2222}}}\right)^\top,$$ and hence,
$$\begin{aligned}
\mathcal{A}(x^0)^4&=2t_0(1-t_0)(2a_{1112}t_0^2+(3a_{1221}+\sqrt{a_{1111}a_{2222}})t_0(1-t_0)\\
&\ \ \ \ +2a_{2221}(1-t_0)^2)>0.
\end{aligned}$$ Namely, $$2a_{1112}t_0^2+(3a_{1221}+\sqrt{a_{1111}a_{2222}})t_0(1-t_0)\\
 +2a_{2221}(1-t_0)^2>0.$$
So, the desired conclusion follows.
\end{proof}

Now we give a simpler sufficient conditions of (strict) copositivity of 4th order 3 dimensional tensors by reducing dimensions.

\begin{theorem} \label{th:39} Let $\mathcal{A}$ be a symmetric tensor of order $4$ and dimension $3$. Assume that
	$$\begin{aligned}
	&a_{1111}\geq 0\ (>0),\ a_{2222}\geq 0\ (>0),\ a_{3333}\geq 0\ (>0)\\
&a_{1123}\geq 0\ (>0),\ a_{1223}\geq 0\ (>0),\ a_{1233}\geq 0\ (>0)\\
&\eta_1=2a_{1112}+\sqrt[4]{a_{1111}^3a_{2222}}\geq 0\ (>0) ,\ \mu_1=2a_{1222}+\sqrt[4]{a_{1111}a_{2222}^3}\geq0\ (>0),\\
&\eta_2=2a_{1113}+\sqrt[4]{a_{1111}^3a_{3333}}\geq 0\ (>0) ,\ \mu_2=2a_{1333}+\sqrt[4]{a_{1111}a_{3333}^3}\geq0\ (>0),\\
&\eta_3=2a_{2223}+\sqrt[4]{a_{2222}^3a_{3333}}\geq 0\ (>0) ,\ \mu_3=2a_{2333}+\sqrt[4]{a_{2222}a_{3333}^3}\geq0\ (>0),\\
&\theta_1=3(2a_{1122}-\sqrt{a_{1111}a_{2222}})+4\sqrt{\eta_1\mu_1} \geq 0\ (>0)\\
&\theta_2=3(2a_{1133}-\sqrt{a_{1111}a_{3333}})+4\sqrt{\eta_2\mu_2} \geq 0\ (>0)\\
&\theta_3=3(2a_{2233}-\sqrt{a_{2222}a_{3333}})+4\sqrt{\eta_3\mu_3} \geq 0\ (>0).
	\end{aligned}$$
	Then $\mathcal{A}$ is (strictly) copositive.
\end{theorem}
\begin{proof} For a vector $x=(x_1,x_2,x_3)^\top$, we have
$$\begin{aligned}\mathcal{A}x^4=&a_{1111}x_1^4+a_{2222}x_2^4+a_{3333}x_3^4+4a_{1222}x_1x_2^3 +4a_{1333}x_1x_3^3\\
		& +4a_{2111}x_1^3x_2+4a_{2333}x_2x_3^3+4a_{3111}x_1^3x_3
		+4a_{3222}x_2^3x_3\\
		&+6a_{1122}x_1^2x_2^2+6a_{1133}x_1^2x_3^2+6a_{2233}x_2^2x_3^2\\
		=&(\frac12a_{1111}x_1^4+4a_{1112}x_1^3x_2+6a_{1122}x_1^2x_2^2+4a_{1222}x_1x_2^3+\frac12a_{2222}x_2^4)\\
+&(\frac12a_{1111}x_1^4+4a_{1113}x_1^3x_3+6a_{1133}x_1^2x_3^2+4a_{1333}x_1x_3^3+\frac12a_{3333}x_3^4)\\
+&(\frac12a_{2222}x_2^4+4a_{2223}x_2^3x_3+6a_{2233}x_2^2x_3^2+4a_{2333}x_2x_3^3+\frac12a_{3333}x_3^4)\\
+& 12a_{1123}x_1^2x_2x_3+12a_{1223}x_1x_2^2x_3 +12a_{1233}x_1x_2x_3^2.
		\end{aligned}$$
Let $$g_1(x_1,x_2)=\frac12a_{1111}x_1^4++4a_{1112}x_1^3x_2+6a_{1122}x_1^2x_2^2+4a_{1222}x_1x_2^3+\frac12a_{2222}x_2^4.$$
Then  $g_1(x_1,x_2)$ may be regarded as a homogeneous polynomial defined by a 4th order 2 dimensional symmetric tensor $\mathcal{B}=(b_{ijkl})$ with its entries
$$b_{1111}=\frac12a_{1111}, b_{1112}=a_{1112}, b_{1122}=a_{1122}, b_{1222}=a_{1222}, b_{2222}=\frac12a_{2222}.$$
That is, $$g_1(x_1,x_2)=\mathcal{B}y^4=\sum_{i,j,k,l=1}^2b_{ijkl}x_ix_jx_kx_l,\mbox{ for all }y=(x_1,x_2)^\top.$$
Then the assumptions imply that
$$\begin{aligned}
&b_{1111}=\frac12a_{1111}\geq 0,\ b_{2222}=\frac12a_{2222}\geq 0,\\
&b_{1112}+\sqrt[4]{b_{1111}^3b_{2222}}=a_{1112}+\frac12\sqrt[4]{a_{1111}^3a_{2222}}=\frac12\eta_1\geq 0 ,\\
&b_{1222}+\sqrt[4]{b_{1111}b_{2222}^3}=a_{1222}+\frac12\sqrt[4]{a_{1111}a_{2222}^3}=\frac12\mu_1\geq0,
	\end{aligned}$$
$$\begin{aligned}&3(b_{1122}-\sqrt{b_{1111}b_{2222}})+4\sqrt{\left(b_{1112}+\sqrt[4]{b_{1111}^3b_{2222}}\right)\left(b_{2221}+\sqrt[4]{b_{1111}b_{2222}^3}\right)}\\
 &=3(a_{1122}-\frac12\sqrt{a_{1111}a_{2222}})+4\sqrt{\frac14\eta_1\mu_1}\\&=\frac12(3(2a_{1122}-\sqrt{a_{1111}a_{2222}})+4\sqrt{\eta_1\mu_1})\\
 &=\frac12\theta_1\geq 0.
\end{aligned}$$
It follows from Theorem \ref{th:37} that the tensor $\mathcal{B}$ is  copositive, that is, $$g_1(x_1,x_2)=\mathcal{B}y^4\geq 0\mbox{ for all }y=(x_1, x_2)^\top\geq 0.$$
Similarly, we also have
$$g_2(x_1,x_3)=\frac12a_{1111}x_1^4+4a_{1113}x_1^3x_3+6a_{1133}x_1^2x_3^2+4a_{1333}x_1x_3^3+\frac12a_{3333}x_3^4\geq0,$$
$$g_3(x_2,x_3)=\frac12a_{2222}x_2^4+4a_{2223}x_2^3x_3+6a_{2233}x_2^2x_3^2+4a_{2333}x_2x_3^3+\frac12a_{3333}x_3^4\geq0.$$
Thus, for all $x=(x_1,x_2,x_3)^\top\geq0,$ we have \begin{align*}\mathcal{A}x^4=&g_1(x_1,x_2)+g_2(x_1,x_3)+g_3(x_2,x_3)\\ &+ 12a_{1123}x_1^2x_2x_3+12a_{1223}x_1x_2^2x_3 +12a_{1233}x_1x_2x_3^2\geq0,\end{align*}
that is, $\mathcal{A}$ is copositive. The proof of strict copositivity of $\mathcal{A}$ is same as the above, we omit it.
\end{proof}	

\begin{remark}
  Theorem \ref{th:39} is proved by reducing dimensions of tensor. That is, an 4th order 3 dimensional tensor is decomposed three 4th order 2 dimensional tensors, and then, analysing the copositivity of these 2 dimensional tensors to obtain the desired sufficient conditions by using Theorem \ref{th:37}. So,  distinctly sufficient conditions may be established by applying Theorems \ref{th:34}, \ref{th:35}, \ref{th:36}, respectively.
\end{remark}
	
\section{\bf Checking vacuum stability of scalar potentials}
\subsection{Vacuum stability of the scalar potential of two real scalars and the Higgs boson}
Recently, Kannike \cite{K2016,K2018} studied the vacuum stability of general scalar potentials
of a few fields. The most general scalar potential of two real scalar fields $\phi_1$ and $\phi_2$ can be expressed as
\begin{equation}\label{eq:41} V(\phi_1,\phi_2)=\lambda_{40}\phi_1^4+\lambda_{31}\phi_1^3\phi_1+\lambda_{22}\phi_1^2\phi_2^2+\lambda_{13}\phi_1\phi_2^3+\lambda_{04}\phi_2^4
=\Lambda \phi^4, \end{equation}
where $\Lambda=(\lambda_{ijkl})$ is the symmetric tensor of scalar couplings and $\phi=(\phi_1,\phi_2)^\top$ is the vector of fields. The tensor of the scalar couplings of the potential is defined by 	
\begin{equation}\label{eq:42}
\Lambda=\left(\begin{aligned}\left(\begin{array}{cc}
\lambda_{40}  & \frac14\lambda_{31}\\
\frac14\lambda_{31} & \frac16\lambda_{22}
\end{array}\right) \left(\begin{array}{cc}
\frac14\lambda_{31} & \frac16\lambda_{22}\\
\frac16\lambda_{22} & \frac14\lambda_{13}
\end{array}
\right) \\
\left(\begin{array}{cc}
\frac14\lambda_{31} & \frac16\lambda_{22}\\
\frac16\lambda_{22} & \frac14\lambda_{13}
\end{array}
\right)
\left(\begin{array}{cc}
\frac16\lambda_{22} & \frac14\lambda_{13}\\
\frac14\lambda_{13} & \lambda_{04}
\end{array}
\right)\end{aligned}\right)
	\end{equation}
that is, $$\begin{aligned}\lambda_{1111}=&\lambda_{40},\ \ \ \lambda_{2222}=\lambda_{04},\\ \lambda_{1112}=&\lambda_{1121}=\lambda_{1211}=\lambda_{2111}=\frac14\lambda_{31},\\ \lambda_{1122}=&\lambda_{1212}=\lambda_{1221}=\lambda_{2112}=\lambda_{2121}=\lambda_{2211}=
\frac16\lambda_{22},\\
 \lambda_{1222}=&\lambda_{2122}=\lambda_{2212}=\lambda_{2221}=\frac14\lambda_{13}. \end{aligned}$$
 It is known that the vacuum stability of the general scalar potential of $2$ real singlet scalar fields is equivalent to the positivity of the polynomial \eqref{eq:41} (\cite{K2016}), i.e., the positive definiteness of the tensor $\Lambda=(\lambda_{ijkl})$.
 Then it follows from  Theorem \ref{th:31} that the tensor $\Lambda$ with $\lambda_{1111}=\lambda_{40}>0$ and $\lambda_{2222}=\lambda_{04}>0$
 is positive definite if and only if (Multiply  by common multiple of denominators to make them simpler),
 $$(a)\ \ \ \left\{\begin{aligned}
&8\lambda_{40}\lambda_{22}-3\lambda_{31}^2\geq0,\\
&4(12\lambda_{40}\lambda_{04}-3\lambda_{31}\lambda_{13}+\lambda_{22}^2)^3\\
&>(72\lambda_{40}\lambda_{22}\lambda_{04}+9\lambda_{31}\lambda_{22}\lambda_{31} -2\lambda_{22}^3-27\lambda_{40}\lambda_{31}^2-27\lambda_{31}^2\lambda_{04})^2
\end{aligned}\right.$$
$$(b)\ \ \ \left\{\begin{aligned}
&8\lambda_{40}\lambda_{22}-3\lambda_{31}^2<0,\\
&16\lambda_{40}^2\lambda_{22}^2+3\lambda_{31}^4+16\lambda_{40}^2\lambda_{31}\lambda_{13}<16\lambda_{40}\lambda_{31}^2\lambda_{22}+64\lambda_{40}^3\lambda_{04},\\
&4(12\lambda_{40}\lambda_{04}-3\lambda_{31}\lambda_{13}+\lambda_{22}^2)^3\\
&>(72\lambda_{40}\lambda_{22}\lambda_{04}+9\lambda_{31}\lambda_{22}\lambda_{31} -2\lambda_{22}^3-27\lambda_{40}\lambda_{31}^2-27\lambda_{31}^2\lambda_{04})^2
\end{aligned}\right.\ \ \ $$
$$(c)\ \ \ \left\{\begin{aligned}
&8\lambda_{40}\lambda_{22}-3\lambda_{31}^2>0,\\
&16\lambda_{40}^2\lambda_{22}^2+3\lambda_{31}^4+16\lambda_{40}^2\lambda_{31}\lambda_{13}=16\lambda_{40}\lambda_{31}^2\lambda_{22}+64\lambda_{40}^3\lambda_{04},\ \ \ \ \ \ \ \ \ \\
&4\lambda_{40}\lambda_{31}\lambda_{22}=8\lambda_{40}^2\lambda_{13}+\lambda_{31}^3\\
&4(12\lambda_{40}\lambda_{04}-3\lambda_{31}\lambda_{13}+\lambda_{22}^2)^3\\
&=(72\lambda_{40}\lambda_{22}\lambda_{04}+9\lambda_{31}\lambda_{22}\lambda_{31} -2\lambda_{22}^3-27\lambda_{40}\lambda_{31}^2-27\lambda_{31}^2\lambda_{04})^2
\end{aligned}\right.$$
Therefore, one of the above cases (1) and (2) and (3)	can guarantee the vacuum stability of the general scalar potential $V(\phi_1,\phi_2)$ of $2$ real singlet scalar fields.\\
 %Kannike \cite{K2016,K2018} gave the conditions of the vacuum stability of $V(\phi_1,\phi_2)$ as follows,
%$$\begin{aligned}
%&\lambda_{40}>0,\ \lambda_{04}>0,\ H>0,\ \Delta>0\\
%\mbox{ or }\ \ &\lambda_{40}>0,\ \lambda_{04}>0,\ \Delta>0, F<0.
%\end{aligned}$$

%\subsection{\bf Vacuum stability for two real scalars and the Higgs boson}

The most general scalar potential of two real scalar fields $\phi_1$
and $\phi_2$ and the Higgs doublet $\mathbf{H}$ (Kannike \cite{K2016,K2018})  is
\begin{align}
V(\phi_1,\phi_2,|H|)=& \lambda_{H}|H|^4+\lambda_{H20}|H|^2\phi_1^2+\lambda_{H11}|H|^2\phi_1\phi_2+\lambda_{H02}|H|^2\phi_2^2\nonumber\\
&\ +\lambda_{40}\phi_1^4+\lambda_{31}\phi_1^3\phi_2+\lambda_{22}\phi_1^2\phi_2^2+\lambda_{13}\phi_1\phi_2^3+\lambda_{04}\phi_2^4,\label{eq:43}\\
=& \lambda_{H}|H|^4+M^2(\phi_1,\phi_2)|H|^2+\bar{V}(\phi_1,\phi_2),\nonumber
\end{align}
where $$M^2(\phi_1,\phi_2)=\lambda_{H20}\phi_1^2+\lambda_{H11}\phi_1\phi_2+\lambda_{H02}\phi_2^2$$ and $$\bar{V}(\phi_1,\phi_2)=V(\phi_1,\phi_2,0)=\lambda_{40}\phi_1^4+\lambda_{31}\phi_1^3\phi_2+\lambda_{22}\phi_1^2\phi_2^2+\lambda_{13}\phi_1\phi_2^3+\lambda_{04}\phi_2^4.$$
Let $x=(\phi_1,\phi_2,|H|)^\top$. Then  $V(\phi_1,\phi_2,|H|)=\mathcal{V}x^4$, where $\mathcal{V}=(v_{ijkl})$ is a 4th order 3  dimensional symmetric tensor with its entries
$$\begin{aligned}v_{1111}=&\lambda_{40},\  v_{2222}=\lambda_{04},\  v_{3333}=\lambda_{H}, v_{1112}=\frac14\lambda_{31},\ v_{1222}=\frac14\lambda_{13},\\
v_{1133}=&\frac16\lambda_{H20},\ v_{1122}= \frac16\lambda_{22},\ v_{2233}=\frac16\lambda_{H02},\\
v_{1233}=&\frac1{12}\lambda_{H11},\ \  v_{ijkl}=0\mbox{ for the others}.
  \end{aligned}$$
Clearly, $\bar{V}(\phi_1,\phi_2)$ is a 4th order 2 dimensional  tensor. Let $\phi=(\phi_1,\phi_2)^\top.$ Then $\bar{V}(\phi_1,\phi_2)=\Lambda \phi^4,$ where $\Lambda$ is a symmetric tensor given by \eqref{eq:42}, which is a principal subtensor of $\mathcal{V}$.
 So, the conditions (a) and (b) and (c) exactly ensure the positive definiteness of $\Lambda$, i.e.,  $\bar{V}(\phi_1,\phi_2)=\Lambda \phi^4>0.$

On the other hand,  $M^2(\phi_1,\phi_2)=\phi^\top M\phi$, where $M$ is a symmtric matrix given by
$$M=\left(\begin{aligned}
&\lambda_{H20} &\ \frac12\lambda_{H11}\\
&\frac12\lambda_{H11} &\ \lambda_{H02}
\end{aligned}\right)$$
It is well-known that $M$ is positive definite if and only if \begin{equation}\label{eq:44}\lambda_{H20}>0,\ \lambda_{H02}>0\mbox{ and }4\lambda_{H20}\lambda_{H02}-\lambda_{H11}^2>0.\end{equation}
So, the positivity of  $V(\phi_1,\phi_2,|H|)$ is made  certain by  $\lambda_{H}>0$ and Eq.\eqref{eq:44} together with the conditions (a) or (b) or (c).

Therefore, the conditions of the vacuum stability for the scalar potential $V(\phi_1,\phi_2,|H|)$ of two real scalar fields $\phi_1$ and $\phi_2$ and the Higgs doublet $\mathbf{H}$ are
$$\lambda_{40}>0,\ \lambda_{04}>0,\ \lambda_{H}>0,\ \lambda_{H20}>0,\ \lambda_{H02}>0,\ 4\lambda_{H20}\lambda_{H02}-\lambda_{H11}^2>0 $$
and  the  inequalities systems (a) or (b) or (c).

\subsection{Vacuum stability for $\mathbb{Z}_3$ scalar dark matter}
Kannike \cite{K2016,K2018} gave another physical example defined by scalar dark matter stable
under a $\mathbb{Z}_3$ discrete group. The most general scalar quartic
potential of the {\bf SM} Higgs $\mathbf{H}_1$, an inert doublet $\mathbf{H}_2$ and a
complex singlet $\mathbf{S}$ which is symmetric under a $\mathbb{Z}_3$ group is
\begin{align}
V(h_1,h_2,s)=& \lambda_1|H_1|^4+\lambda_2|H_2|^4+\lambda_3|H_1|^2|H_2|^2+\lambda_4(H_1^{\dagger} H_2)(H_2^{\dagger} H_1)\nonumber\\
&\ +\lambda_S|S|^4+\lambda_{S1}|S|^2|H_1|^2+\lambda_{S2}|S|^2|H_2|^2\nonumber\\
&\ +\frac12(\lambda_{S12}S^2H_1^{\dagger} H_2+\lambda_{S12}^*S^{{\dagger}2}H_2^{\dagger} H_1)\nonumber\\
= & \lambda_1h_1^4+\lambda_2h_2^4+\lambda_3h_1^2h^2_2+\lambda_4\rho^2 h_1^2h_2^2\nonumber\\
&\ +\lambda_Ss^4+\lambda_{S1}s^2h_1^2+\lambda_{S2}s^2h_2^2-|\lambda_{S12}|\rho s^2h_1h_2\nonumber\\
\equiv & \lambda_Ss^4+M^2(h_1,h_2)s^2+\hat{V}(h_1,h_2),\nonumber
\end{align}
Where \begin{align}
h_1=|H_1|,\ h_2&=|H_2|, H_2^{\dagger} H_1=h_1h_2\rho e^{i\phi}, S=se^{i\phi_S}, \lambda_{S12}=-|\lambda_{S12}|,\nonumber\\
M^2(h_1,h_2)=&\lambda_{S1}h_1^2+\lambda_{S2}h_2^2-|\lambda_{S12}|\rho h_1h_2,\label{eq:45}\\
\hat{V}(h_1,h_2)=&V(h_1,h_2,0)=\lambda_1h_1^4+\lambda_2h_2^4+\lambda_3h_1^2h^2_2+\lambda_4\rho^2 h_1^2h^2_2.\label{eq:46}
\end{align}
The orbit space parameter $\rho\in[0,1]$ as implied by the Cauchy inequality $0\leq|H^\dagger_1 H_2|\leq|H_1||H_2|$.

Let $x=(h_1,h_2,s)^\top$. Then $V(h_1,h_2,s)= \mathcal{V}x^4$,  where $\mathcal{V}=(v_{ijkl})$ is an 4th order 3 dimensional real symmetric tensor given by
\begin{align}
&v_{1111}=\lambda_1,\ v_{2222}=\lambda_2,\ v_{3333}=\lambda_S, \nonumber\\
&v_{1122}=\frac16(\lambda_3+\lambda_4\rho^2),\ v_{1133}=\frac16\lambda_{S1},\ v_{2233}=\frac16\lambda_{S2}, \nonumber\\
& v_{1233}=-\frac1{12}|\lambda_{S12}|\rho,\ v_{ijkl}=0 \mbox{ for the others}.\nonumber
\end{align}
It follows from Theorem \ref{th:33} that the conditions of strict copositivity of the tensor $\mathcal{V}$ (that is, $V(h_1,h_2,s)= \mathcal{V}x^4>0$) are given by
\begin{align}
&\lambda_1>0,\ \lambda_2>0,\ \lambda_S>0, \nonumber\\
&\lambda_3+\lambda_4\rho^2>0,\ \lambda_{S1}>0,\ \lambda_{S2}>0, \nonumber\\
& -|\lambda_{S12}|\rho+\sqrt{2\lambda_{S1}\lambda_{S2}}>0,\nonumber\\
&\sqrt{\lambda_S\lambda_{S1}\lambda_{S2}}-|\lambda_{S12}|\rho\sqrt{\lambda_{S}}
 +\sqrt{2\lambda_{S}\sqrt{\lambda_{S1}\lambda_{S2}}(-|\lambda_{S12}|\rho+\sqrt{2\lambda_{S1}\lambda_{S2}})}>0.\nonumber
\end{align}
So the above conditions assure the potential $V(h_1,h_2,s)$ symmetric under a $\mathbb{Z}_3$ group is bounded from below. These conditions are different from ones of Kannike \cite{K2016,K2018} and Chen-Huang-Qi \cite{CHQ2018}.

\section*{Acknowledgments} The authors would like to express their sincere thanks to Professor Yimin Wei, Professor Chen Ling, Professor Gaohang Yu for their constructive comments and valuable suggestions.
	% ----------------------------------------------------------------
\bibliographystyle{amsplain}

\end{document}